\documentclass[11pt]{amsart}

\usepackage[margin=1.2in]{geometry}
\usepackage{amsmath,amssymb}
\usepackage[colorlinks]{hyperref}

\newcommand{\chiB}{\chi_{\mathrm{B}}}

\newcommand{\chimu}{\chi_{\mu}}
\newcommand{\chiM}{\chi_{\mathrm{M}}}
\newcommand{\Z}{\mathbb{Z}}
\newcommand{\N}{\mathbb{N}}

\newtheorem{theorem}{Theorem}
\newtheorem{lemma}[theorem]{Lemma}
\theoremstyle{remark}
\newtheorem{remark}[theorem]{Remark}

\title{The measure half of the $2n+1$ problem}

\author{Jos\'e de Jes\'us Pelayo G\'omez}
\thanks{Independent researcher. Email: \texttt{pelayuss@gmail.com}.}

\date{July 20, 2026}

\subjclass[2020]{03E15; 05C15, 28A05}
\keywords{Borel chromatic number; measurable chromatic number;
descriptive graph combinatorics; graphs generated by functions}

\begin{document}

\begin{abstract}
Let $F_0,\dots,F_{n-1}$ be Borel functions on a standard Borel space
$X$ and let $G_F$ be the graph they generate. We prove that for every
Borel probability measure $\mu$ on $X$ there are a forward-invariant
$\mu$-conull Borel set $A\subseteq X$ and a proper Borel
$(2n+1)$-coloring of $G_F\restriction A$, and that consequently
$\chiM(G_F)\le2n+1$ in the total sense of Kechris and Marks. This
answers the measure half of \cite[Problem~5.14]{KM} positively, for
every $n$ and with the optimal constant; no invariance, no local
finiteness, and no finiteness of the Borel chromatic number are
assumed. The argument is elementary and specific to measure; the
corresponding Borel problem and the Baire-measurable half remain
open.
\end{abstract}

\maketitle

\section{Statement and context}

Throughout, $X$ is a standard Borel space, $n\ge1$, and
$F_0,\dots,F_{n-1}\colon X\to X$ are Borel functions. The graph $G_F$
generated by the family joins $x\ne y$ whenever $F_i(x)=y$ or
$F_i(y)=x$ for some $i<n$. For a Borel probability measure $\mu$ on
$X$, the \emph{$\mu$-measurable chromatic number} $\chimu(G_F)$ is the
least cardinality of a standard Borel space $Y$ for which there is a
$\mu$-measurable proper coloring $c\colon X\to Y$ --- measurable with
respect to the $\mu$-completion of the Borel $\sigma$-algebra, and
proper on all of $X$, the \emph{total} sense of \cite[p.~18]{KM} ---
and
$\chiM(G_F)=\sup_\mu\chimu(G_F)$.

Classically $\chi(G_F)\le2n+1$: the orientation $x\to F_i(x)$ has
out-degree at most $n$, so every finite subgraph has a vertex of total
degree at most $2n$, and greedy coloring and De Bruijn--Erd\H{o}s
compactness \cite{dBE} apply. The bound is attained: the rotational
family $F_i(x)=x+(i+1)\bmod(2n+1)$ on $\Z_{2n+1}$ generates the
complete graph $K_{2n+1}$ \cite[p.~48]{KM}. Kechris, Solecki, and
Todor\v{c}evi\'c asked whether $\chiB(G_F)<\aleph_0$ implies
$\chiB(G_F)\le2n+1$ \cite[Question~4.9]{KST}; by Marks's determinacy
method, the value $2n+1$ is attained even by acyclic graphs generated
by $n$ Borel bijections \cite{MarksJAMS}. Kechris and Marks raised
the measurable and Baire-measurable halves explicitly
\cite[Problem~5.14]{KM}: is $\chiM(G_F)\le2n+1$? Toward the measure
half, the known bounds are quadratic in general (Miller, Palamourdas;
\cite[Theorem~5.13(iv)]{KM}, \cite{Pal}) and $8$ for $n=3$
\cite[Theorem~5.13(iii)]{KM}; for locally finite graphs, Conley and
Miller obtained the bound $4n+1$ in the Baire-measurable setting, and
in the measurable setting when the connectedness equivalence relation
is moreover hyperfinite \cite{CM}. The optimal constant $2n+1$ was
previously known only in special cases:
for commuting families \cite[Corollary~3.22]{MP} and, very recently,
for probability-measure-preserving graphs generated by bounded-to-one
functions \cite{HST}. See also the survey \cite{MarksSurvey}.

Here we prove the measure half in full.

\begin{theorem}\label{thm:main}
Let $n\ge1$, let $X$ be a standard Borel space, and let
$F_0,\dots,F_{n-1}\colon X\to X$ be Borel functions. For every Borel
probability measure $\mu$ on $X$ there are a forward-invariant
$\mu$-conull Borel set $A\subseteq X$ and a Borel proper
$(2n+1)$-coloring of $G_F\restriction A$. Moreover
$\chimu(G_F)\le2n+1$ in the total sense of \cite[p.~18]{KM}, and
consequently
\[
   \chiM(G_{F_0,\dots,F_{n-1}})\le 2n+1 .
\]
Thus the measure half of \cite[Problem~5.14]{KM} has a positive
answer for every $n$.
\end{theorem}

\begin{remark}[optimality and scope]
By the rotational example above, $\chiM=2n+1$ is attained already by
finite instances, so the constant is optimal. The content of
Theorem~\ref{thm:main} is definable uniformity: nothing is assumed
about the functions (fixed points, non-injectivity, and infinite
in-degree are allowed), about the measure (no quasi-invariance), or
about $\chiB(G_F)$.
\end{remark}

\begin{remark}[the subinvariant case]\label{rem:pmp}
If $\sum_{i<n}(F_i)_*\mu\le L\mu$ for some $L<n+1$ --- for instance
when every $F_i$ preserves $\mu$, with $L=n$ --- one may take
$\nu=\mu$: only the envelope lemma (Lemma~\ref{lem:envelope}) is then
unnecessary, and the rest of the proof applies verbatim.
\end{remark}

\begin{remark}[scope of the method]
The proof uses $\sigma$-additivity, Radon--Nikodym derivatives, and
Borel--Cantelli in an essential way. We do not know a corresponding
category argument; in particular, this note does not address the
Baire-measurable half of \cite[Problem~5.14]{KM} or the Borel problem
of \cite[Question~4.9]{KST}, both of which remain open.
\end{remark}

\section{The proof}

Fix $n$, $X$, and $F_0,\dots,F_{n-1}$ as above, and put $q=2n+1$.
Fixed points of the $F_i$ produce no edges, so an index $i$ is ignored
at $x$ whenever $F_i(x)=x$; \emph{image} always means a nontrivial
image $F_i(x)\ne x$. All measures are Borel measures on $X$, and
$(F)_*\lambda(B)=\lambda(F^{-1}(B))$ denotes the pushforward, a
monotone and positively linear operation on finite measures. For a
Borel map $c\colon X\to\Z_q$ let
\[
 E(c)=\bigl\{x:\exists i<n\
   [F_i(x)\ne x\ \text{ and }\ c(x)=c(F_i(x))]\bigr\}
\]
be its Borel set of \emph{bad sources}. Every monochromatic edge has
the form $\{x,F_i(x)\}$ with $x\in E(c)$, so $c$ is proper if and only
if $E(c)=\varnothing$; the same applies on any forward-invariant
subset.

\begin{lemma}[countable palette; essentially \cite{KST}]
\label{lem:countable}
$G_F$ admits a proper Borel coloring $\eta\colon X\to\N$.
\end{lemma}

\begin{proof}
Fix a countable algebra $\{A_k\}_{k\in\N}$ of Borel subsets of $X$,
containing $X$, closed under complements and finite intersections, and
separating points (for instance, the algebra generated by a countable
basis of a compatible Polish topology). Given $x$, its images form a
finite set disjoint from $\{x\}$; separation and closure under
complements and intersections provide a set in the algebra containing
$x$ and none of its images (the set $X$ itself when there are no
images). Let $\eta(x)$ be the least $k$ such that $x\in A_k$ and
$F_i(x)\notin A_k$ for every $i<n$ with $F_i(x)\ne x$. Each condition
is Borel, so $\eta$ is Borel. If $y=F_i(x)\ne x$ and
$\eta(x)=\eta(y)=k$, then $y\notin A_k$ by the choice at $x$, while
$y\in A_k$ by the choice at $y$ --- a contradiction.
\end{proof}

\begin{lemma}[thinning]\label{lem:thinning}
Let $\nu$ be a finite Borel measure on $X$ and let $p\in(0,1)$,
$\gamma=p(1-p)^n$. Every Borel set $E\subseteq X$ contains a
$G_F$-independent Borel set $C$ with $\nu(C)\ge\gamma\nu(E)$.
\end{lemma}

\begin{proof}
Let $\eta$ be as in Lemma~\ref{lem:countable}, let
$\beta$ be the Bernoulli$(p)$ product measure on $\Omega=2^{\N}$, and
for $\omega\in\Omega$ put
\[
 C_\omega=\bigl\{x\in E:\
   \omega_{\eta(x)}=1\ \text{and}\
   \omega_{\eta(F_i(x))}=0
   \text{ for every }i<n\text{ with }F_i(x)\ne x\bigr\}.
\]
The set $\{(x,\omega):x\in C_\omega\}\subseteq X\times\Omega$ is
Borel: the maps $(x,\omega)\mapsto\omega_{\eta(x)}$ and
$(x,\omega)\mapsto\omega_{\eta(F_i(x))}$ are Borel, and membership is
a finite Boolean combination of Borel conditions on their values; in
particular each $C_\omega$ is Borel. Each $C_\omega$ is independent:
any edge inside $C_\omega$ has the form $\{x,y\}$ with
$y=F_i(x)\ne x$, and membership of $x$ requires $\omega_{\eta(y)}=0$
while membership of $y$ requires $\omega_{\eta(y)}=1$. For fixed
$x\in E$, properness of $\eta$ makes $\eta(x)$ different from the
colors of the images of $x$, of which there are at most $n$ distinct
values; the corresponding coordinates of $\omega$ are independent, so
\[
   \beta(\{\omega:x\in C_\omega\})\ge p(1-p)^n=\gamma .
\]
By Fubini's theorem applied to $\nu\times\beta$,
\[
 \int\nu(C_\omega)\,d\beta(\omega)
 =\int_E\beta(\{\omega:x\in C_\omega\})\,d\nu(x)
 \ge\gamma\nu(E),
\]
so some realization $C=C_\omega$ satisfies
$\nu(C)\ge\gamma\nu(E)$.
\end{proof}

\begin{lemma}[one recoloring step]\label{lem:recoloring}
Let $\nu$ be a Borel probability measure on $X$ with
\begin{equation}\label{eq:subinvariant}
   \sum_{i<n}(F_i)_*\nu\le L\nu
   \qquad\text{for some }L<n+1,
\end{equation}
and put
\[
   \kappa=\frac{L}{n+1}<1,\qquad
   \gamma=\frac1{n+1}\Bigl(\frac{n}{n+1}\Bigr)^{n},\qquad
   \delta=\gamma(1-\kappa)>0 .
\]
For every Borel map $c\colon X\to\Z_q$ there is a Borel map
$c'\colon X\to\Z_q$ such that $\{c'\ne c\}$ is a $G_F$-independent
Borel subset of $E(c)$ and
\begin{equation}\label{eq:contraction}
   \nu(E(c'))\le(1-\delta)\,\nu(E(c)).
\end{equation}
\end{lemma}

\begin{proof}
For $a\in\Z_q$ define the finite Borel measure
\[
 \rho_a(B)=
 \sum_{i<n}\nu\bigl(\{y:c(y)=a,\ F_i(y)\ne y,\ F_i(y)\in B\}\bigr)
 \qquad(B\subseteq X\text{ Borel}).
\]
For every Borel $B$,
\[
 \sum_{a\in\Z_q}\rho_a(B)
 \le\sum_{i<n}\nu(F_i^{-1}(B))
 =\sum_{i<n}(F_i)_*\nu(B)
 \le L\nu(B)
\]
by \eqref{eq:subinvariant}; in particular $\rho_a\ll\nu$. Choose
nonnegative Borel versions $h_a$ of the Radon--Nikodym derivatives
$d\rho_a/d\nu$ and redefine them to be $0$ on the common $\nu$-null
Borel set where $\sum_a h_a>L$, so that
\begin{equation}\label{eq:density-sum}
   \sum_{a\in\Z_q}h_a(x)\le L
   \qquad\text{for every }x\in X,
\end{equation}
while the identities $\rho_a(B)=\int_B h_a\,d\nu$ persist.

Apply Lemma~\ref{lem:thinning} with $p=1/(n+1)$ to $E(c)$ and choose
an independent Borel $C\subseteq E(c)$ with
\begin{equation}\label{eq:C-large}
   \nu(C)\ge\gamma\nu(E(c)).
\end{equation}
For $x\in C$ let
\[
   \mathcal A(x)=
   \Z_q\setminus
   \{c(F_i(x)):i<n,\ F_i(x)\ne x\},
\]
the set of colors available at $x$; at most $n$ colors are excluded,
so $|\mathcal A(x)|\ge q-n=n+1$. Let $a(x)$ be the first element of
$\mathcal A(x)$, in the fixed order of $\Z_q$, minimizing $h_a(x)$
over $a\in\mathcal A(x)$; this is a Borel choice, and since the $h_a$
are nonnegative, \eqref{eq:density-sum} gives
\begin{equation}\label{eq:kappa}
   h_{a(x)}(x)
   \le\frac1{n+1}\sum_{a\in\mathcal A(x)}h_a(x)
   \le\frac{L}{n+1}
   =\kappa
   \qquad(x\in C).
\end{equation}
Define $c'=a$ on $C$ and $c'=c$ off $C$. Every point of $C$ changes
color: $x\in E(c)$ means that some image of $x$ carries $c(x)$, so
$c(x)\notin\mathcal A(x)$; thus $\{c'\ne c\}=C$.

We claim
\begin{equation}\label{eq:bad-inclusion}
   E(c')\subseteq(E(c)\setminus C)\cup N,
   \qquad
 N=\bigcup_{i<n}
 \{y:F_i(y)\ne y,\ F_i(y)\in C,\
             c(y)=a(F_i(y))\}.
\end{equation}
Indeed, no point of $C$ lies in $E(c')$: by independence no image of
$x\in C$ lies in $C$, so the images of $x$ keep their colors, and
$c'(x)=a(x)\in\mathcal A(x)$ avoids all of them. If $y\notin C$ and
$y\in E(c')$ via the index $i$, then $c(y)=c'(y)=c'(F_i(y))$ with
$F_i(y)\ne y$; if $F_i(y)\notin C$ this color equals $c(F_i(y))$ and
$y\in E(c)\setminus C$, while if $F_i(y)\in C$ it equals $a(F_i(y))$
and $y\in N$.

Partitioning according to the value $b=c(y)$, which on the $i$-th set
in \eqref{eq:bad-inclusion} equals $a(F_i(y))$, the union bound and
the definition of $\rho_b$ give
\begin{align*}
 \nu(N)
 &\le
 \sum_{i<n}\nu\bigl(
   \{y:F_i(y)\ne y,\ F_i(y)\in C,\ c(y)=a(F_i(y))\}\bigr)\\
 &=\sum_{b\in\Z_q}\sum_{i<n}\nu\bigl(
   \{y:c(y)=b,\ F_i(y)\ne y,\ F_i(y)\in C\cap\{a=b\}\}\bigr)\\
 &=\sum_{b\in\Z_q}
   \rho_b\bigl(C\cap\{x:a(x)=b\}\bigr)
 =\int_C h_{a(x)}(x)\,d\nu(x)
 \le\kappa\,\nu(C),
\end{align*}
the last inequality by \eqref{eq:kappa}. Combining with
\eqref{eq:bad-inclusion}, \eqref{eq:C-large}, and $C\subseteq E(c)$,
\[
 \nu(E(c'))
 \le \nu(E(c))-\nu(C)+\kappa\nu(C)
 \le\bigl(1-\gamma(1-\kappa)\bigr)\nu(E(c)),
\]
which is \eqref{eq:contraction}.
\end{proof}

\begin{lemma}[envelope]\label{lem:envelope}
Let $\mu$ be a Borel probability measure on $X$, let
$\frac{n}{n+1}<r<1$, and put
\[
   P_*\lambda=\frac1n\sum_{i<n}(F_i)_*\lambda,
   \qquad
   \nu=(1-r)\sum_{k=0}^{\infty}r^kP_*^k\mu .
\]
Then $\nu$ is a Borel probability measure with:
\begin{enumerate}
\item $\nu\ge(1-r)\mu$; in particular $\mu\ll\nu$;
\item $\sum_{i<n}(F_i)_*\nu\le L\nu$ with $L=\frac nr<n+1$;
\item more generally, every finite Borel measure $\lambda$ with
$\sum_{i<n}(F_i)_*\lambda\le L\lambda$ satisfies
$(F_w)_*\lambda\le L^{|w|}\lambda$ for every composition
$F_w=F_{w_1}\circ\dots\circ F_{w_k}$ of the $F_i$; in particular, by
(2), every such $F_w$ is $\nu$-nonsingular.
\end{enumerate}
\end{lemma}

\begin{proof}
Each $P_*^k\mu$ is a Borel probability measure and the weights
$(1-r)r^k$ sum to $1$, so $\nu$ is a Borel probability measure, and
(1) is the $k=0$ term. By monotone convergence, pushforward commutes
with the series, so
\[
   P_*\nu
   =(1-r)\sum_{k=0}^{\infty}r^kP_*^{k+1}\mu
   =\frac{\nu-(1-r)\mu}{r}
   \le\frac1r\,\nu,
\]
and $\sum_{i<n}(F_i)_*\nu=nP_*\nu\le L\nu$; moreover $L=n/r<n+1$
because $r>\frac{n}{n+1}$. For (3), each single
$(F_i)_*\lambda\le\sum_{j<n}(F_j)_*\lambda\le L\lambda$, and by
induction on the length of $w$, using monotonicity of the pushforward,
\[
 (F_{w_1}\circ F_{w_2}\circ\dots\circ F_{w_k})_*\lambda
 =(F_{w_1})_*\bigl((F_{w_2}\circ\dots\circ F_{w_k})_*\lambda\bigr)
 \le L^{k-1}(F_{w_1})_*\lambda
 \le L^{k}\lambda .
\]
Nonsingularity means $\lambda(B)=0\Rightarrow\lambda(F_w^{-1}(B))=0$,
immediate from $(F_w)_*\lambda\le L^{|w|}\lambda$.
\end{proof}

\begin{proof}[Proof of Theorem~\ref{thm:main}]
Fix $\mu$, fix $r\in(\frac n{n+1},1)$, and let $\nu$ and $L<n+1$ be
as in Lemma~\ref{lem:envelope}. (When
$\sum_i(F_i)_*\mu\le L\mu$ already holds with some $L<n+1$ --- for
instance when every $F_i$ preserves $\mu$, with $L=n$ --- take
$\nu=\mu$ instead; see Remark~\ref{rem:pmp}.)

\emph{Step 1: iteration and stabilization.}
Let $c_0\equiv0$ and let $c_{t+1}$ be obtained from $c_t$ by
Lemma~\ref{lem:recoloring}; write $E_t=E(c_t)$ and
$C_t=\{c_{t+1}\ne c_t\}\subseteq E_t$. The constants $\gamma$,
$\kappa$, $\delta$ do not depend on $t$, so \eqref{eq:contraction}
iterates to $\nu(E_t)\le(1-\delta)^t$, whence
\[
   \sum_{t\ge0}\nu(E_t)<\infty,
   \qquad
   \sum_{t\ge0}\nu(C_t)<\infty .
\]
By Borel--Cantelli, $\nu$-almost every $x$ lies in only finitely many
$E_t$, hence in only finitely many $C_t$, so $(c_t(x))_t$ changes only
finitely often. Let
\[
 S=\bigl\{x:(c_t(x))_t\ \text{is eventually constant and}\
                  x\notin E_t\ \text{for all sufficiently large }t\bigr\},
\]
a $\nu$-conull Borel set, and let $c_\infty(x)=\lim_tc_t(x)$ on $S$;
this map is Borel, since for every $a\in\Z_q$
\[
 \{x\in S:c_\infty(x)=a\}
 =S\cap\bigcup_{N\in\N}\bigcap_{t\ge N}\{x:c_t(x)=a\}.
\]
If
$x,F_i(x)\in S$ with $F_i(x)\ne x$ and
$c_\infty(x)=c_\infty(F_i(x))$, then after both sequences stabilize
$x\in E_t$ for every large $t$, contradicting $x\in S$; so
$c_\infty$ is proper on $G_F\restriction S$.

\emph{Step 2: the forward core.}
Let $(F_w)_{w\in M}$ enumerate the countable monoid generated by the
$F_i$ (including the identity) and put
\[
   A=\bigcap_{w\in M}F_w^{-1}(S)\subseteq S .
\]
By Lemma~\ref{lem:envelope}(3) each $F_w^{-1}(S)$ is $\nu$-conull, so
$A$ is a Borel $\nu$-conull set; since $\mu\le(1-r)^{-1}\nu$ by
Lemma~\ref{lem:envelope}(1), $A$ is also $\mu$-conull. It is
forward-invariant by construction, and
$c_\infty\restriction A$ is a Borel proper $q$-coloring of
$G_F\restriction A$, proving the first assertion of the theorem.

\emph{Step 3: total measurable extension.}
For $x\in X\setminus A$ let $O(x)$ be the set of images of $x$ lying
in $A$ and
\[
 \mathcal L(x)=\Z_q\setminus c_\infty[O(x)],
 \qquad
 |\mathcal L(x)|\ge q-|O(x)|\ge q-n=n+1 .
\]
Forward-invariance of $A$ implies that every edge crossing between
$A$ and its complement originates in $X\setminus A$; hence any proper
coloring $\varphi$ of $G_F\restriction(X\setminus A)$ with
$\varphi(x)\in\mathcal L(x)$ combines with $c_\infty\restriction A$
to a proper coloring of all of $X$.

Such a $\varphi$ exists. First, every finite
$T\subseteq X\setminus A$ contains a vertex $x$ with
$\deg_T(x)+|O(x)|\le2n$, where $\deg_T$ is the degree in
$G_F\restriction T$: writing $\operatorname{out}_T(x)$ for the set of
images of $x$ in $T$, each edge of $G_F\restriction T$ is an
out-image realization for at least one of its endpoints and
contributes $2$ to $\sum_{x\in T}\deg_T(x)$, so
\[
 \sum_{x\in T}
 \bigl(\deg_T(x)+|O(x)|\bigr)
 \le
 2\sum_{x\in T}|\operatorname{out}_T(x)|
 +\sum_{x\in T}|O(x)|
 \le 2n|T|,
\]
since the images of $x$ in $T$ and in $A$ are distinct points, at
most $n$ in all. Removing such vertices successively and coloring
them in reverse order --- at its turn, $x$ must avoid at most
$\deg_T(x)<|\mathcal L(x)|$ colors inside $\mathcal L(x)$ --- gives a
proper coloring of $G_F\restriction T$ from the lists. Now the space
$\prod_{x\in X\setminus A}\mathcal L(x)$ is compact by Tychonoff's
theorem, and for finite $T$ the set $K_T$ of assignments proper on
$G_F\restriction T$ is closed (it constrains finitely many
coordinates) and nonempty; since
$K_{T\cup T'}\subseteq K_T\cap K_{T'}$, the family has the finite
intersection property, and any $\varphi\in\bigcap_TK_T$ is proper on
$G_F\restriction(X\setminus A)$, in the spirit of \cite{dBE}.

Let $\bar c$ be $c_\infty$ on $A$ and $\varphi$ on $X\setminus A$.
Each color class
$\bar c^{-1}(b)=(c_\infty^{-1}(b)\cap A)\cup\varphi^{-1}(b)$ is a
Borel set modified inside the $\mu$-null set $X\setminus A$, hence
$\mu$-measurable. Thus $\bar c$ is a total, $\mu$-measurable, proper
$q$-coloring, so $\chimu(G_F)\le q=2n+1$. Since $\mu$ was arbitrary,
$\chiM(G_F)\le2n+1$.
\end{proof}

\begin{remark}
The proof gives more than stated: the coloring on the conull core is
the pointwise limit of a sequence of Borel recolorings, the
bad sets decay exponentially, and the constants depend only on $n$
and on the choice of $r$.
\end{remark}



\begin{thebibliography}{Ma16}

\bibitem[CM]{CM}
C.~T. Conley and B.~D. Miller,
\emph{A bound on measurable chromatic numbers of locally finite Borel
graphs}, Math. Res. Lett. \textbf{23} (2016), 1633--1644.

\bibitem[dBE]{dBE}
N.~G. de Bruijn and P.~Erd\H{o}s,
\emph{A colour problem for infinite graphs and a problem in the theory
of relations}, Indag. Math. \textbf{13} (1951), 369--373.

\bibitem[HST]{HST}
C.~Higgins, P.~Spaas, and A.~Tenenbaum,
\emph{Spectral theory for Borel pmp graphs},
arXiv:2602.05185, 2026.

\bibitem[KM]{KM}
A.~S. Kechris and A.~S. Marks,
\emph{Descriptive graph combinatorics}, manuscript, 2020.
\url{https://math.berkeley.edu/~marks/papers/combinatorics20book.pdf}

\bibitem[KST]{KST}
A.~S. Kechris, S.~Solecki, and S.~Todor\v{c}evi\'c,
\emph{Borel chromatic numbers}, Adv. Math. \textbf{141} (1999), 1--44.

\bibitem[Ma16]{MarksJAMS}
A.~S. Marks,
\emph{A determinacy approach to Borel combinatorics},
J. Amer. Math. Soc. \textbf{29} (2016), 579--600.

\bibitem[Ma23]{MarksSurvey}
A.~S. Marks,
\emph{Measurable graph combinatorics}, in: \emph{Proceedings of the
International Congress of Mathematicians 2022}, Vol.~3, EMS Press,
2023, pp.~1488--1502.

\bibitem[MP]{MP}
C.~Meehan and K.~Palamourdas,
\emph{Borel chromatic numbers of graphs of commuting functions},
Fund. Math. \textbf{253} (2021), 219--237.

\bibitem[Pal]{Pal}
K.~Palamourdas,
\emph{$1,2,3,\dots,2n+1,\infty!$}, Ph.D. thesis, University of
California, Los Angeles, 2012.
\url{https://escholarship.org/uc/item/1075v66x}

\end{thebibliography}
\end{document}